\documentclass[letterpaper,oneside,english]{amsart}
\usepackage{ae}
\usepackage{aecompl}
\usepackage[T1]{fontenc}
\usepackage[latin1]{inputenc}
\usepackage{amssymb}

\makeatletter
 \theoremstyle{plain}
\newtheorem{thm}{Theorem}[section]
  \theoremstyle{definition}
  \newtheorem{defn}[thm]{Definition}
  \theoremstyle{plain}
  \newtheorem{lem}[thm]{Lemma}
  \theoremstyle{plain}
  \newtheorem{prop}[thm]{Proposition}
  \theoremstyle{plain}
  \newtheorem{cor}[thm]{Corollary}

\usepackage[dvips]{hyperref}

\usepackage{mathbbol}

\usepackage[all,dvips,arc,curve,rotate]{xy}

\usepackage{ifthen}

\newboolean{Details}

\setboolean{Details}{false}

\makeatother
\begin{document}

\title{From Matrix to Operator Inequalities}

\author{Terry A. Loring}

\address{Department of Mathematics and Statistics, University of New Mexico,
Albuquerque, NM 87131, USA.}



\urladdr{http://www.math.unm.edu/\textasciitilde{}loring/}

\maketitle
\markright {$C^*$-Algebra Relations (preprint version)}\markright {$C^*$-Algebra Relations (preprint version)}

\begin{center}
\ifthenelse{\boolean{Details}}{{\Large Fully Detailed Version}}{}
\par\end{center}

\begin{abstract}
We generalize L\"{o}wner's method for proving that matrix monotone
functions are operator monotone. The relation $x\leq y$ on bounded
operators is our model for a definition for $C^{*}$-relations of
being residually finite dimensional.

Our main result is a meta-theorem about theorems involving relations
on bounded operators. If we can show there are residually finite dimensional
relations involved, and verify a technical condition, then such a
theorem will follow from its restriction to matrices.

Applications are shown regarding norms of exponentials, the norms
of commutators and ``positive'' noncommutative $*$-polynomials.
\end{abstract}

\section{Introduction}

Our topic is bounded operators that satisfy relations that involve
algebraic relations, the operator norm, functional calculus and positivity.
The word positive, when applied to matrices, shall mean positive semidefinite. 

The $*$-strong topology can bridge the gap between representations
of relations by bounded operators on Hilbert space and representations
by matrices. This feature of the $*$-strong topology has been noted
before, for example by L\"{o}wner in \cite{BendatSherman}, or in
the context of residually finite dimensional $C^{*}$-algebras in
\cite{ExelLoringRFDfreeProd}.

This article is essentially independent of our previous paper \cite{LoringCstarRelations}
on $C^{*}$-relations. We minimize the role of universal $C^{*}$-algebras.
Perhaps someone will see how to strip out the $C^{*}$-algebras and
get a result that works for norms other than the operator norm.

Theorem~\ref{thm:RFDimpliesClosedReduction}, our main result, a
meta-theorem. We first define $C^{*}$-relation, and define on $C^{*}$-relations
the concepts of closed and residually finite dimensional (RFD). The
meta-theorem is that, given a theorem about matrices which states
that an RFD $C^{*}$-relation implies a closed $C^{*}$-relation,
we may conclude the same implication holds for all bounded operators.

The author is grateful for inspiring conversations with R.\ Bhatia,
T.\ Shulman, V.\ Shulman and S.\ Silvestrov.

\section{$C^{*}$-Relations}

\begin{defn}
Suppose $\mathcal{X}$ is a set. A statement $R$ about functions
$f:\mathcal{X}\rightarrow A$ into various $C^{*}$-algebras is a
\emph{$C^{*}$-relation} if the following four axioms hold:
\begin{description}
\item [{R1}] the unique function $\mathcal{X}\rightarrow\{0\}$ satisfies
$R;$
\item [{R2}] if $\varphi:A\hookrightarrow B$ is an injective $*$-homomorphism
and $\varphi\circ f:\mathcal{X}\rightarrow B$ satisfies $R$ for
some $f$ then $f:\mathcal{X}\rightarrow A$ also satisfies $R;$
\item [{R3}] if $\varphi:A\rightarrow B$ is a $*$-homomorphism and $f:\mathcal{X}\rightarrow A$
satisfies $R$ then $\varphi\circ f:\mathcal{X}\rightarrow B$ also
satisfies $R;$
\item [{R4f}] if $f_{j}:\mathcal{X}\rightarrow A_{j}$ satisfy $R$ for
$j=1,\ldots n$ the so does $f=\prod f_{j}$ where\[
f:\mathcal{X}\rightarrow\prod_{j=1}^{n}A_{j}\]
sends $x$ to $\left\langle f_{1}(x),\ldots,f_{n}(x)\right\rangle .$
\end{description}
\end{defn}
Examples of $C^{*}$-relations include the zero-sets of $*$-polynomials
in noncommuting variables, henceforth called NC $*$-polynomials.

When $\mathcal{X}=\{1,2,\ldots,m\},$ the case we really care about,
we use $a_{1},\ldots,a_{m}$ in place of the function notation $j(q)=a_{q}.$
Given a NC $*$-polynomial $p(x_{1},\ldots,x_{m})$ with constant
term zero, its zero-set is the $C^{*}$-relation\[
p(a_{1},\ldots,a_{m})=0.\]
Other $C^{*}$-relations associated to $p$ include\[
p(a_{1},\ldots,a_{m})\geq0\]
and\[
\left\Vert p(a_{1},\ldots,a_{m})\right\Vert \leq C\]
for a constant $C>0,$ as well as\[
\left\Vert p(a_{1},\ldots,a_{m})\right\Vert <C.\]

Given a set $\mathcal{R}$ of $C^{*}$-relations, a function $f:\mathcal{X}\rightarrow A$
to a $C^{*}$-algebra $A$ is called a \emph{representation} \emph{of}
$\mathcal{R}$ \emph{in} $A$ if every statement in $\mathcal{R}$
is true for $f.$ A function $\iota:\mathcal{X}\rightarrow U$ into
a $C^{*}$-algebra is \emph{universal} for $\mathcal{R}$ if $\iota$
is a representation of $\mathcal{R}$ and for every representation
$f:\mathcal{X}\rightarrow A$ of $\mathcal{R}$ there is a unique
$*$- homomorphism $\varphi:U\rightarrow A$ so that $\varphi\circ\iota=f.$

It is important to note that often there is no universal $C^{*}$-algebra
and no universal representation. See \cite{LoringCstarRelations}.

We use the notation $C^{*}\left\langle \mathcal{X}\left|\mathcal{R}\right.\right\rangle $
for $U$ and call it the universal $C^{*}$-algebra. Notice that universal
representation $\iota$ is usually what we should be talking about.
Notice also that $\iota$ need not be injective, but still we often
say that the representation $f:\mathcal{X}\rightarrow A$ of $\mathcal{R}$
extends to a unique $*$- homomorphism $\varphi:U\rightarrow A$ with
the requirement $\varphi(\iota(x))=f(x).$ A good exercise is to show
that $\iota(\mathcal{X})$ must generate $C^{*}\left\langle \mathcal{X}\left|\mathcal{R}\right.\right\rangle $
as a $C^{*}$-algebra. Alternately, this is clear from the proof of
Theorem 2.6 in \cite{LoringCstarRelations}.

Given a set $\mathcal{R}$ of $C^{*}$-relations on a set $\mathcal{X},$
we let $\textnormal{rep}_{\mathcal{R}}(\mathcal{X},A)$ denote the
set of all representations of $\mathcal{X}$ in $A.$ If $\mathbb{H}$
is a Hilbert space then we set \[
\textnormal{rep}_{\mathcal{R}}(\mathcal{X},\mathbb{H})=\textnormal{rep}_{\mathcal{R}}(\mathcal{X},\mathbb{B}(\mathbb{H})).\]

The notation\[
\prod_{\lambda\in\Lambda}A_{\lambda}\]
shall denote the $C^{*}$-algebra product consisting of all bounded
sequences or families $\left\langle a_{\lambda}\right\rangle _{\lambda\in\Lambda}$
that have $a_{\lambda}$ in $A_{\lambda}.$ Given a family of functions
$f_{\lambda}:\mathcal{X}\rightarrow A_{\lambda}$ we say it is \emph{bounded}
if $\sup_{\lambda}\left\Vert f_{\lambda}(x)\right\Vert $ is finite
for all $x$ in $\mathcal{X}.$ For such a bounded family we define
their \emph{product} to be the function \[
\prod_{\lambda\in\Lambda}f_{\lambda}:\mathcal{X}\rightarrow\prod_{\lambda\in\Lambda}A_{\lambda}\]
that sends $x$ to the family $\left\langle f_{\lambda}(x)\right\rangle .$ 

It could be argued that the above be called the sum of the representations,
c.f.\ Section II.6.1 in \cite{BlackadarOperatorAlgebras}. Notice
that given Hilbert spaces $\mathbb{H}_{\lambda}$ we have the inclusion
(block diagonal) \[
\prod_{\lambda}\mathbb{B}\left(\mathbb{H}_{\lambda}\right)\subseteq\mathbb{B}\left(\bigoplus_{\lambda}\mathbb{H}_{\lambda}\right).\]

A potential source of confusion is that we are talking about representations
of relations in $C^{*}$-algebras, but then often want to represent
various $C^{*}$-algebras on Hilbert space. Sometime we cut out the
middle man and talk of representations of relations on Hilbert space.
However, what we allow to be called $C^{*}$-relations on a set of
operators are properties that can be determined by how they sit in
the $C^{*}$-algebra they generate.

\begin{defn}
Suppose $\mathcal{R}$ is a set of $C^{*}$-relations on a set $\mathcal{X}.$
We say $\mathcal{R}$ is \emph{closed} if
\begin{description}
\item [{R4b}] for every bounded family $f_{\lambda}:\mathcal{X}\rightarrow A_{\lambda}$
of representations of $\mathcal{R}$ the product function $\prod f_{\lambda}$
is also a representation of $\mathcal{R}.$
\end{description}
We say $\mathcal{R}$ is \emph{compact} if it satisfies the stronger
axiom
\begin{description}
\item [{R4}] every family $f_{\lambda}:\mathcal{X}\rightarrow A_{\lambda}$
of representations of $\mathcal{R}$ is bounded and the associated
product function $\prod f_{\lambda}$ is a representation of $\mathcal{R}.$
\end{description}
Following Hadwin, Kaonga, Mathes (\cite{Hadwin-Kaonga-Mathes}) Phillips
(\cite{PhillipsProCstarAlg}) and others, we showed in \cite{LoringCstarRelations}
that $\mathcal{R}$ is compact if and only if there is a universal
$C^{*}$-algebra for $\mathcal{R}.$
\end{defn}
For example, \[
\left\{ x^{2}-x=0,\ x^{*}-x=0\right\} \]
is compact, and has universal $C^{*}$-algebra isomorphic to $\mathbb{C}.$
Also compact is \[
\left\{ \left\Vert x^{2}-x\right\Vert \leq\frac{1}{8},\ x^{*}-x=0\right\} \]
and this also has a universal $C^{*}$-algebra that is commutative.
On the other hand\[
\left\{ x^{2}-x=0\right\} \]
is closed but not compact, while\[
\left\{ \left\Vert x^{2}-x\right\Vert <\frac{1}{8},\ x^{*}-x=0\right\} \]
is not even closed.

It is sometimes easier to look at unital relations and unital $C^{*}$-algebras.
Everything here carries over. Notice we are not putting $1$ in $\mathcal{X},$
but use it symbolically in relations to stand for the unit in $A$
when considering a function $f:\mathcal{X}\rightarrow A.$ Rather
than introduce notation for this, we limit our discussion of unital
relations to quoting other papers.

The relations associated to NC $*$-polynomials are not the only interesting
$C^{*}$-relations. The relation\[
0\leq x\leq1\]
is compact, with universal $C^{*}$-algebra $C_{0}(0,1].$ Another
relation on $\{ x,y,z\}$ is\[
0\leq\left[\begin{array}{cc}
y & x^{*}\\
x & z\end{array}\right],\]
which is to be interpreted so that $X,$ $Y$ and $Z$ in a $C^{*}$-algebra
$A$ form a representation if and only if the matrix\[
\left[\begin{array}{cc}
Y & X^{*}\\
X & Z\end{array}\right]\]
is a positive element in $\mathbf{M}_{2}(A).$

Many results about operators relative to the operator norm, or positivity,
can be stated in the form where one set of $C^{*}$-relations implies
another. For example\[
x^{*}x=xx^{*}\implies\left[\begin{array}{cc}
\left|x\right| & x^{*}\\
x & \left|x\right|\end{array}\right]\geq0\]
or\[
0\leq k\leq1,\left\Vert h\right\Vert \leq1,\left\Vert hk-kh\right\Vert \leq\epsilon\implies\left\Vert hk^{\frac{1}{2}}-k^{\frac{1}{2}}h\right\Vert \leq\frac{5}{4}\epsilon\]
or\[
x^{*}x=xx^{*},\ yx=xy\implies yx^{*}=x^{*}y.\]
In many cases, such a theorem will follow from the special case of
that theorem restricted to the matrix case. Whether this leads to
a new result, a shorter proof of a known result, or a new but harder
proof of a known result, depends on the example. What seems interesting
is just how many theorems in the literature involve $C^{*}$-relations
that are residually finite dimensional.

\section{Residually Finite Dimensional $C^{*}$-Relations}

We certainly want to look at relations that are compact and have universal
$C^{*}$-algebra that is residually finite dimensional (RFD). For
example, for some $0<\epsilon\leq2$ the relations\[
u^{*}u=uu^{*}=v^{*}v=vv^{*}=1,\ \left\Vert uv-vu\right\Vert \leq\epsilon\]
have universal $C^{*}$-algebra that has been dubbed ``the soft torus''
by Exel, and this $C^{*}$-algebra has been shown to be residually
finite dimensional by Eilers and Exel in \cite{EilersExelSoftTorusRFD}. 

A $C^{*}$-algebra is \emph{residually finite dimensional} if there
is a separating family of representations of $A$ on finite dimensional
Hilbert spaces. 

The restriction to compact relations is artificial in operator theory.
A very important example is the relation $\left\Vert xy+yx\right\Vert \leq\epsilon$
on $\{ x,y\}.$ (The obvious ``and'' operation turns a set of $C^{*}$-relations
into a single relation, so we tend to use relation and set of relations
interchangeably.) We could discuss RFD $\sigma$-$C^{*}$-algebras,
but prefer to take as our starting point an alternate characterization
of RFD $C^{*}$-algebras in \cite{ExelLoringRFDfreeProd}.

On $\textnormal{rep}_{\mathcal{R}}(\mathcal{X},\mathbb{H})$ we will
consider the pointwise $*$-strong topology, whenever $\mathbb{H}$
is a Hilbert space and $\mathcal{R}$ is a $C^{*}$-relation on a
set $\mathcal{X}.$ 

Compare this to \[
\textnormal{rep}(A,\mathbb{H})=\left\{ \pi:A\rightarrow\mathbb{B}(\mathbb{H})\left|\ \pi\mbox{ is a }*\mbox{-homomorhpism}\right.\right\} \]
 for a $C^{*}$-algebra $A$ with the pointwise $*$-strong topology.
Equivalently, consider this with the pointwise strong topology. A
representation $\pi$ is said to be \emph{finite dimensional} if its
essential subspace is finite dimensional. The relevant result from
\cite{ExelLoringRFDfreeProd} is that $A$ is RFD if and only if for
all $\mathbb{H}$ the finite dimensional representations are dense
in $\textnormal{rep}(A,\mathbb{H}).$ For more characterizations of
a $C^{*}$-algebra being RFD see \cite{ArchboldRFD}.

We need a definition of finite dimensional for $f\in\textnormal{rep}_{\mathcal{R}}(\mathcal{X},\mathbb{H}).$
We define the \emph{essential subspace} of $f$ to be\[
\left\{ \xi\in\mathbb{H}\left|f(x)\xi=\left(f(x)\right)^{*}\xi=0,\ \forall x\in\mathcal{X}\right.\right\} ^{\perp}.\]
We say $f$ is \emph{finite dimensional} if its essential subspace
is finite dimensional. Notice this property has nothing to do with
$\mathcal{R}.$

If $\mathcal{R}$ is a compact set of $C^{*}$-relations then the
essential space of $f$ is the same as the essential space of the
associated representation $\pi$ of $C^{*}\left\langle \mathcal{X}\left|\mathcal{R}\right.\right\rangle .$
Thus $f$ is finite dimensional if and only if $\pi$ is finite dimensional.

\begin{defn}
A set of $C^{*}$-relations $\mathcal{R}$ on $\mathcal{X}$ is \emph{residually
finite dimensional (RFD)} if there are finite constants \[
C(x,r)\quad(x\in X,\ r\in[0,\infty))\]
 so that, for every $\mathbb{H}$ and any choice of nonnegative constants
$r_{x},$ every $f\in\textnormal{rep}_{\mathcal{R}}(\mathcal{X},\mathbb{H})$
satisfying\[
\| f(x)\|\leq r_{x}\quad(\forall x\in\mathcal{X})\]
is in the $*$-strong closure of \[
\left\{ g\in\textnormal{rep}_{\mathcal{R}}(\mathcal{X},\mathbb{H})\left|\ g\mbox{ is finite dimensional and }g(x)\leq C(x,r_{x})\ \forall x\in\mathcal{X}\right.\right\} .\]

\end{defn}
We say $f$ in $\textnormal{rep}_{\mathcal{R}}(\mathcal{X},\mathbb{H})$
is \emph{cyclic} if there is a vector $\xi$ so that $A\xi$ is dense
in $\mathbb{H},$ where $A=C^{*}(f(\mathcal{X})).$ Even if $A\xi$
is not dense, we call its closure a \emph{cyclic subspace} for $f.$
We say $f$ is \emph{unitarily equivalent} to $g$ in $\textnormal{rep}_{\mathcal{R}}(\mathcal{X},\mathbb{K})$
if there is a unitary $U:\mathbb{K}\rightarrow\mathbb{H}$ so that
$g(x)=U^{-1}f(x)U.$ If $\mathbb{K}$ is a reducing subspace for all
operators in $f(\mathcal{X})$ then let $U$ be the inclusion of $\mathbb{K}$
in $\mathbb{H}$ and let $g(x)=U^{*}f(x)U.$ By \textbf{R3} this is
also a representation and we call $g$ a \emph{subrepresentation of}
$f.$

\begin{lem}
Suppose $\mathcal{R}$ is a set of $C^{*}$-relations on $\mathcal{X}.$
Every Hilbert space representation of $\mathcal{R}$ is unitarily
equivalent to a product of cyclic representations.
\end{lem}
\begin{proof}
\ifthenelse{\boolean{Details}}{ Suppose $f$ is a representation
of $\mathcal{R}$ on $\mathbb{H}.$ Consider the set of all sets of
pairwise orthogonal cyclic subspaces. Inclusion makes this a partial
order that is not empty; the set containing the zero subspace qualifies.
Upper bounds for chains clearly exist, so Zorn's Lemma tells us there
is a maximal set of orthogonal cyclic subspaces. If the union of these
subspaces is not dense, the orthogonal complement is a nontrivial
reducing subspace. Take any nonzero vector there, form the the cyclic
subspace it determines and add it to this set. This is a larger qualified
set, giving us a contradiction.

Since $\mathbb{H}$ is the internal direct sum of cyclic subspaces
for $f,$ each of which is of course reducing, we find $f$ is unitarily
equivalent to the product of cyclic representations.

}{The proof is almost identical to that of the same result for representations
of $C^{*}$-algebras.}
\end{proof}
\begin{lem}
\label{lem:RFDapproxCyclicSuffices} A set $\mathcal{R}$ of $C^{*}$-relations
is RFD if and only if there are finite constants $C(x,r)$ for $x\in X$
and $r\in[0,\infty))$ so that, for every $\mathbb{H},$ every cyclic
$f\in\textnormal{rep}_{\mathcal{R}}(\mathcal{X},\mathbb{H})$ is in
the $*$-strong closure of \[
\left\{ g\in\textnormal{rep}_{\mathcal{R}}(\mathcal{X},\mathbb{H})\left|\ g\mbox{ is finite dimensional and }g(x)\leq C(x,\| f(x)\|)\right.\right\} .\]

\end{lem}
\begin{proof}
The forward implication is obvious, so assume the condition on the
cyclic representations holds for some choice of $C(x,r).$

We may as well assume\[
f=\prod_{\gamma\in\Gamma}f_{\gamma}\]
where $f_{\lambda}$ is a cyclic representation on $\mathbb{H}_{\gamma}$
and $\mathbb{H}=\bigoplus\mathbb{H}_{\gamma}.$ Suppose $\epsilon>0$
and $\xi=\left\langle \xi_{\gamma}\right\rangle $ is a unit vector.
There is a finite set $\Gamma_{0}$ so that when we define $\eta=\left\langle \eta_{\gamma}\right\rangle $
by\[
\eta_{\gamma}=\left\{ \begin{array}{cc}
\xi_{\gamma} & \mbox{if }\gamma\in\Gamma_{0}\\
0 & \mbox{if }\gamma\notin\Gamma_{0}\end{array}\right.\]
we have $\left\Vert \xi-\eta\right\Vert \leq\delta$ for \[
\delta=\frac{\epsilon}{2}\left(\| f(x)\|+C(x,\| f(x)\|)\right)^{-1}.\]
(Without loss of generality, $C(x,r)\neq0.$) Suppose $\Gamma_{0}$
has $q$ elements. For each $\gamma$ in $\Gamma_{0}$ there is a
finite dimensional representation $g_{\gamma}:\mathcal{X}\rightarrow\mathbb{H}_{\gamma}$
so that\[
\left\Vert g_{\gamma}(x)\right\Vert \leq C(x,\| f(x)\|)\]
and\[
\left\Vert g_{\gamma}(x)\xi_{\gamma}-f_{\gamma}(x)\xi_{\gamma}\right\Vert ,\left\Vert \left(g_{\gamma}(x)\right)^{*}\xi_{\gamma}-\left(f_{\gamma}(x)\right)^{*}\xi_{\gamma}\right\Vert \leq\frac{\epsilon}{2q}.\]
For $\gamma\notin\Gamma_{0}$ set $g_{\gamma}(x)=0.$ Let $g=\prod_{\gamma}g_{\gamma},$
which is a representation, first in $\prod_{\gamma\in\Gamma_{0}}\mathbb{B}(\mathbb{H}_{\gamma})$
by \textbf{R4f}, and then on $\mathbb{H}$ by \textbf{R3}. It satisfies
the norm condition since\[
\left\Vert g(x)\right\Vert =\sup_{\gamma}\left\Vert g_{\gamma}(x)\right\Vert .\]
The essential space of $g$ is just the sum of the orthogonal essential
spaces of the $g_{\gamma}$ for $\gamma\in\Gamma_{0}$ and so $g$
is also finite dimensional. For each $x$ we have\begin{eqnarray*}
\left\Vert g(x)\xi-f(x)\xi\right\Vert  & \leq & \left\Vert g(x)-f(x)\right\Vert \left\Vert \xi-\eta\right\Vert +\left\Vert g(x)\eta-f(x)\eta\right\Vert \\
 & \leq & \left(\| f(x)\|+C(x,\| f(x)\|)\right)\delta+\sum_{\gamma\in\Gamma_{0}}\left\Vert g_{\gamma}(x)\xi_{\gamma}-f_{\gamma}(x)\xi_{\gamma}\right\Vert \\
 & \leq & \left(\| f(x)\|+C(x,\| f(x)\|)\right)\delta+q\rho\\
 & = & \epsilon\end{eqnarray*}
and \ifthenelse{\boolean{Details}}{\begin{eqnarray*}
\left\Vert \left(g(x)\right)^{*}\xi-\left(f(x)\right)^{*}\xi\right\Vert  & \leq & \left\Vert g(x)-f(x)\right\Vert \left\Vert \xi-\eta\right\Vert +\left\Vert \left(g(x)\right)^{*}\eta-\left(f(x)\right)^{*}\eta\right\Vert \\
 & \leq & \left(\| f(x)\|+C(x,\| f(x)\|)\right)\delta+\sum_{\gamma\in\Gamma_{0}}\left\Vert \left(g_{\gamma}(x)\right)^{*}\xi_{\gamma}-\left(f_{\gamma}(x)\right)^{*}\xi_{\gamma}\right\Vert \\
 & \leq & \left(\| f(x)\|+C(x,\| f(x)\|)\right)\delta+q\rho\\
 & = & \epsilon.\end{eqnarray*}
}{ similarly\[
\left\Vert \left(g(x)\right)^{*}\xi-\left(f(x)\right)^{*}\xi\right\Vert \leq\epsilon.\]
}
\end{proof}
\begin{prop}
If $\mathcal{R}$ is a compact set of $C^{*}$-relations on a set
$\mathcal{R}$ then $\mathcal{R}$ is RFD if and only if $C^{*}\left\langle \mathcal{X}\left|\mathcal{R}\right.\right\rangle $
is RFD.
\end{prop}
\begin{proof}
By the discussion above, this follows directly from Theorem 2.4 in
\cite{ExelLoringRFDfreeProd}.
\end{proof}
Every $C^{*}$-algebra is isomorphic to the universal $C^{*}$-algebra
of some $C^{*}$-relations, c.f.\  Section 2 in \cite{LoringCstarRelations}.
There is an abundant supply of RFD $C^{*}$-algebras and so an abundant
supply of RFD $C^{*}$-relations. Examples include the subhomogeneous
$C^{*}$-algebras.

Given a specific $C^{*}$-algebra it can be difficult to find a nice
universal set of generator and relations. Conversely, given a set
of $C^{*}$-relations, it can be difficult to get a description of
its universal $C^{*}$-algebra that is more useful than the given
universal property. For present purposes it is best to work directly
with representations of $C^{*}$-relations.

\begin{lem}
Suppose $\mathcal{\mathcal{R}}$ is a set of $C^{*}$-relations $\mathcal{R}$
on $\mathcal{X}.$ If $\mathcal{R}$ is closed and $\left\langle f_{\lambda}\right\rangle _{\lambda\in\Lambda}$
is a bounded net in $\textnormal{rep}(A,\mathbb{H})$ that converges
to the function $f:\mathcal{X}\rightarrow\mathbb{B}(\mathbb{H})$
then $f\in\textnormal{rep}(A,\mathbb{H}).$ 
\end{lem}
\begin{proof}
The key is noticing inside\[
\prod_{\lambda\in\Lambda}\mathbb{B}(\mathbb{H})\]
the $C^{*}$-algebra $A$ of all bounded nets $\left\langle a_{\lambda}\right\rangle $
indexed by $\Lambda$ that have $*$-strong limits\[
L\left(\left\langle a_{\lambda}\right\rangle \right)=\lim_{\lambda}a_{\lambda}\quad(*\mbox{-strong}).\]
Recall, say from \cite[I.3.2.1]{BlackadarOperatorAlgebras}, that
we need boundedness to gain joint continuity of multiplication in
the $*$-strong topology. Here we let $\lambda$ range over the directed
set $\Lambda$ that indexes the net $f_{\lambda}.$

The $f_{\lambda}$ form a bounded family of representations, so $\prod f_{\lambda}$
determines a representation of $\mathcal{R}$ in $A.$ Now we use
the naturality property for $C^{*}$-relations and conclude $f=L\circ\prod f_{\lambda}$
is a representation.
\end{proof}
Now the main theorem.

\begin{thm}
\label{thm:RFDimpliesClosedReduction} Suppose $\mathcal{R}$ and
$\mathcal{S}$ are $C^{*}$-relations on $\mathcal{X}.$ If $\mathcal{R}$
is residually finite dimensional and $\mathcal{S}$ is closed, and
if every finite-dimensional representation of $\mathcal{R}$ is a
representation of $\mathcal{S},$ then every representation of $\mathcal{R}$
is a representation of $\mathcal{S}.$
\end{thm}
\begin{proof}
Given a representation $f:\mathcal{X}\rightarrow\mathbb{B}(\mathbb{H})$
of $\mathcal{R},$ the fact that $\mathcal{R}$ is RFD tells us that
there are functions $f_{\lambda}:\mathcal{X}\rightarrow\mathbb{B}(\mathbb{H})$
with finite dimensional essential spaces that are representations
of $\mathcal{R}$ and so that $f_{\lambda}(x)$ converges $*$-strongly
to $f(x).$ By assumption, the $f_{\lambda}$ are also representation
of $\mathcal{S}$ and, since $\mathcal{S}$ is closed, we conclude
that $f$ is a representation of $\mathcal{S}.$
\end{proof}

\section{Examples}

It is easy to find lots of closed $C^{*}$-relations. Here are enough
to keep us busy while we attack the harder problem of finding RFD
$C^{*}$-relations. Blackadar noted in \cite{Blackadar-shape-theory}
the importance of ``softening'' a relation $p(x_{1},\ldots,x_{n})=0$
to $\| p(x_{1},\ldots,x_{n})\|\leq\epsilon.$ 

\begin{prop}
\label{pro:NCpolyRelationsAreClosed} Suppose $\epsilon>0$ is real
number. If $p$ is a NC $*$-polynomial in $x_{1},\ldots,x_{n},$
with constant term zero, then each of\[
p\left(x_{1},\ldots,x_{n}\right)=0,\]
\[
p\left(x_{1},\ldots,x_{n}\right)\geq0,\]
\[
\left\Vert p\left(x_{1},\ldots,x_{n}\right)\right\Vert \leq\epsilon\]
is a closed $C^{*}$-relation. 
\end{prop}
\begin{proof}
Consider the last relation, for illustration. The other parts of the
proof are similar.

Certainly \[
\left\Vert p\left(0,0,0,\ldots,0\right)\right\Vert =\left\Vert 0\right\Vert =0\leq\epsilon.\]

The evaluation of a NC $*$-polynomial does not depend on the ambient
algebra. Therefore axiom \textbf{R2} holds.

NC $*$-polynomials are natural, so \textbf{R3} holds.

Given families $\left\langle x_{j}^{(\lambda)}\right\rangle $ with
\[
\sup_{\lambda}\left\Vert x_{j}^{(\lambda)}\right\Vert <\infty\]
we immediately get elements in the product $C^{*}$-algebra $\prod A_{\lambda}.$
All NC $*$-polynomials respect products, so\begin{eqnarray*}
 &  & \left\Vert p\left(\left\langle x_{1}^{(\lambda)}\right\rangle ,\left\langle x_{2}^{(\lambda)}\right\rangle ,\left\langle x_{3}^{(\lambda)}\right\rangle ,\ldots,\left\langle x_{n}^{(\lambda)}\right\rangle \right)\right\Vert \\
 & = & \left\Vert \left\langle p\left(x_{1}^{(\lambda)},x_{2}^{(\lambda)},x_{3}^{(\lambda)},\ldots,x_{n}^{(\lambda)}\right)\right\rangle \right\Vert \\
 & = & \sup_{\lambda}\left\Vert p\left(x_{1}^{(\lambda)},x_{2}^{(\lambda)},x_{3}^{(\lambda)},\ldots,x_{n}^{(\lambda)}\right)\right\Vert \\
 & \leq & \epsilon.\end{eqnarray*}
Therefore \textbf{R4b} holds. 
\end{proof}
The relation that inspired this study is $x\leq y,$ which is easily
shown to be RFD using an approximate unit. 

\begin{prop}
If $f$ is a continuous, real-valued function on $[0,\infty),$ then\[
\left\{ x^{*}=x,\ y^{*}=y,\ f(x)\leq f(y)\right\} \]
is a closed set of $C^{*}$-relations. 
\end{prop}
\begin{proof}
This follows from trivial facts such as $f(0)\leq f(0)$ and well-known
facts about the functional calculus.
\end{proof}
\begin{prop}
The relation $x=x^{*}$ is closed $C^{*}$-relation.
\end{prop}
\begin{proof}
This a special case of Proposition~\ref{pro:NCpolyRelationsAreClosed}
\end{proof}
\begin{prop}
The relation $0\leq x$ is closed $C^{*}$-relation. 
\end{prop}
\begin{proof}
A key fact about $C^{*}$-algebras is that positivity ($x=y^{*}y$
for some $y$) does not  depend on the ambient $C^{*}$-algebra either.
See, for example, Section 1.6.5 in \cite{DixmierCstarAlgebras}.
\end{proof}
\begin{prop}
If $f$ is a holomorphic function on the complex plane and $\epsilon$
is real number such that $\epsilon\geq|f(0)|$ then $\left\Vert f(x)\right\Vert \leq\epsilon$
is a closed $C^{*}$-relation.
\end{prop}
\begin{proof}
The functional calculus is to be applied in the unitization of the
ambient $C^{*}$-algebra. We know $f(0)=f(0)\mathbb{1}$ for the zero
operator and so $\{0\}$ is a representation. The holomorphic functional
calculus is natural, does not depend on the surrounding $C^{*}$-algebra,
and respects finite products since polynomial do.
\end{proof}
\begin{prop}
The union of two closed sets of $C^{*}$-relations on the same set
is closed. If a set of $C^{*}$-relations on a set $\mathcal{X}$
is closed, the same properties applied to a larger set $\mathcal{Y}\supseteq\mathcal{X}$
form a closed set of $C^{*}$-relations.
\end{prop}
\begin{proof}
These statements should be obvious.
\end{proof}
\begin{prop}
\label{pro:fractionalNCpolysAreClosed} If $p$ is a NC $*$-polynomial
in $x_{1},\ldots,x_{n}$ with constant term zero, and give $t_{1},\ldots,t_{m}$
positive exponents and possibly repeated indices $j_{1},\ldots,j_{m},$
then \[
\left\{ x_{j}^{*}=x_{j},\ p\left(x_{j_{1}}^{t_{1}},\ldots,x_{j_{m}}^{t_{m}}\right)\geq0\right\} \]
is a closed set of $C^{*}$-relations. If only integer powers of $x_{r}$
are used the relation $x_{r}^{*}=x_{r}$ may be dropped and the result
is still a closed set of $C^{*}$-relations.
\end{prop}
\begin{proof}
Pile the functional calculus higher and deeper. Just to illustrate,
we are talking about a relation such as $x^{\frac{1}{3}}yx^{\frac{2}{3}}\geq0$
applied to pairs $(x,y)$ of operators where $x$ is positive.
\end{proof}
Harder is the task of finding RFD relations. We start with the classic
that kicked off this investigation.

\begin{thm}
The set of $C^{*}$-relations\[
\left\{ x^{*}=x,\ y^{*}=y,\ x\leq y\right\} \]
 is RFD.
\end{thm}
\begin{proof}
We dress L\"{o}wner's argument (\cite{BendatSherman}) in categorical
clothing.

By Lemma~\ref{lem:RFDapproxCyclicSuffices} we need only consider
representations on a separable Hilbert space $\mathbb{H}$. Suppose
$x$ and $y$ are bounded operators on $\mathbb{H}$ that are self-adjoint
and that $x\leq y.$ Let $p_{n}$ be the projection onto the first
$n$ elements in some fixed orthonormal basis. Define $x_{n}=p_{n}xp_{n}$
and $y_{n}=p_{n}yp_{n}$ so that $x_{n}$ and $y_{n}$ have norms
bounded by $\| x\|$ and $\| y\|$ and $0\leq x_{n}\leq y_{n}.$ The
operators $x$ and $y$ form finite dimensional relations, and they
converge $*$-strongly to $x$ and $y.$ 
\end{proof}
\begin{thm}
\label{lem:rightBoundOnRealPartIsRFD} Suppose $\beta>0$ is a real
number. The $C^{*}$-relation $\textnormal{Re }x\leq\beta$ is RFD.
\end{thm}
\begin{proof}
This is almost identical to the last proof. \ifthenelse{\boolean{Details}}{We
adopt that notation. Given $x$ on $\mathbb{H}$ with $x+x^{*}\leq2\beta$
let $x_{n}=p_{n}xp_{n}.$ Then \begin{eqnarray*}
x_{n}+x_{n}^{*} & = & p_{n}(x+x^{*})p_{n}\\
 & \leq & p_{n}(2\beta)p_{n}\\
 & \leq & 2\beta\end{eqnarray*}
 and $\| x_{n}\|\leq\left\Vert x\right\Vert .$}{}
\end{proof}
\begin{thm}
Suppose $\beta>1$ is a real number. The $C^{*}$-relation $\left\Vert e^{\textnormal{Re }x}\right\Vert \leq\beta$
is RFD.
\end{thm}
\begin{proof}
Since the real part of $x$ is Hermitian, the exponential of the real
part is hermitian with spectrum contained in the positive real line.
Therefore $\left\Vert e^{\textnormal{Re }x}\right\Vert \leq e^{\beta}$
holds if and only if $e^{\textnormal{Re }x}\leq e^{\beta}$ which
holds if and only if $\textnormal{Re }x\leq\ln(\beta).$ This relation
has the same representations as the relation in Lemma~\ref{lem:rightBoundOnRealPartIsRFD},
so must itself be RFD.
\end{proof}
\begin{thm}
The empty set of relations on any set $\mathcal{X}$ is a closed $C^{*}$-relation.
\end{thm}
\begin{proof}
This is more-or-less contained in \cite{GoodearlMenalRFDandFree}.
Given a Hilbert space $\mathbb{H}$ and operators $a_{j}$ on $\mathbb{H}$
with $j\in\mathcal{X},$ take any net of finite-rank projections $p_{\lambda}$
converging strongly to the identity. Then $p_{\lambda}a_{j}p_{\lambda}$
is bounded in norm by $\| a_{j}\|$ and converges $*$-strongly to
$a_{j}.$ 
\end{proof}
Many amalgamated products $A\ast_{C}B$ turn out to be RFD when $A$
and $B$ are RFD. The simplest theorem of this sort, proved in \cite{ExelLoringRFDfreeProd},
is that $A$ and $B$ being RFD implies $A\ast B$ is RFD. This generalizes
easily here to something very useful.

\begin{thm}
Suppose $\mathcal{X}$ and $\mathcal{Y}$ are disjoint sets. If $\mathcal{R}$
is an RFD set of $C^{*}$-relations on $\mathcal{X}$ and $\mathcal{S}$
is an RFD closed set of $C^{*}$-relations on $\mathcal{Y}$ then,
regarding both sets as relations on $\mathcal{X}\cup\mathcal{Y},$
the set $\mathcal{R}\cup\mathcal{S}$ is an RFD set of $C^{*}$-relations.
\end{thm}
\begin{proof}
All we need to know is that if $A$ is a set of operators on $\mathbb{H}$
that are zero on the orthogonal complement of the finite dimensional
subspace $\mathbb{H}_{1},$ and if $B$ is a set of operators on $\mathbb{H}$
that are zero on the orthogonal complement of the finite dimensional
subspace $\mathbb{H}_{2},$ then the union is a set of operators that
is zero on the orthogonal complement of the subspace $\mathbb{H}_{1}+\mathbb{H}_{2},$
which is also finite dimensional. 
\end{proof}
\begin{prop}
\label{pro:softNCpolyRelationsAreRFD} If $p_{1},\ldots,p_{m}$ are
NC $*$-polynomials in $x_{1},\ldots,x_{n},$ are homogeneous of degrees
that can vary, and $\epsilon_{s}>0$ are real constants and $0\leq n_{1}\leq n_{2}\leq n$
then\begin{eqnarray*}
 &  & 0\leq x_{j},\quad(j=1,\ldots,n_{1})\\
 &  & x_{j}^{*}=x_{j},\quad(j=n_{1}+1,\ldots,n_{2})\\
 &  & \left\Vert p_{s}\left(x_{1},\ldots,x_{n}\right)\right\Vert \leq\epsilon_{s}\quad(s=1,\ldots,m)\end{eqnarray*}
 form a closed set of $C^{*}$-relations. 
\end{prop}
\begin{proof}
Assume $x_{1},\ldots,x_{n}$ are in $\mathbb{B}(\mathbb{H})$ where
$\mathbb{H}$ is separable and that these satisfy the above relations.
Let $u_{k}$ be a countable approximate identity for the compact operators,
with $0\leq u_{k}\leq1,$ that is quasi-central for $x_{1},\ldots,x_{n}.$
Such an approximate identity exists by Corollary~3.12.16 in \cite{Pedersen-C*-algebrasBook}.
Applying a decreasing perturbation to each $u_{k}$ we may further
assume each $u_{k}$ is finite-rank.

Let $x_{j,k}=u_{k}x_{j}u_{k}.$ Clearly $\| x_{j,k}\|\leq\| x_{j}\|$
and $0\leq x_{j,k}$ for $j\leq n_{1}$ and $x_{j,k}^{*}=x_{j,k}$
for $n_{1}<j\leq n_{2}.$ Also $x_{j,k}\rightarrow x_{j}$ in the
$*$-strong topology. For fixed $k$ the $x_{j,k}$ all act as zero
on the complement of range of $u_{k},$ which is finite dimensional.
However, we need to modify the $x_{j,k}$ to make the last line of
relations hold.

Suppose $p_{s}$ is homogeneous of degree $d_{s}.$ This means $u_{k}$
appears $2d_{s}$ times in each monomial in $p_{s}.$ Since $u_{k}$
is quasi-central for the $x_{j}$ we have\[
\lim_{k\rightarrow\infty}\left\Vert p_{s}\left(x_{1,k},\ldots,x_{n,k}\right)-u_{k}^{2d_{s}}p_{s}\left(x_{1},\ldots,x_{n}\right)\right\Vert =0.\]
Therefore\begin{eqnarray*}
\limsup_{k\rightarrow\infty}\left\Vert p_{s}\left(x_{1,k},\ldots,x_{n,k}\right)\right\Vert  & = & \limsup_{k\rightarrow\infty}\left\Vert u_{k}^{2d_{s}}p_{s}\left(x_{1},\ldots,x_{n}\right)\right\Vert \\
 & \leq & \left\Vert p_{s}\left(x_{1},\ldots,x_{n}\right)\right\Vert .\end{eqnarray*}
It is easy to show that if $a_{\lambda}\rightarrow a$ in the strong
topology then\[
\liminf_{\lambda\rightarrow\infty}\left\Vert a_{\lambda}\right\Vert \geq\left\Vert a\right\Vert .\]
 The $x_{j,k}$ are bounded sequences converging to $x_{j}$ in the
$*$-strong topology, so\[
\lim_{k}p_{s}\left(x_{1,k},\ldots,x_{n,k}\right)=p_{s}\left(x_{1},\ldots,x_{n}\right)\quad(*\mbox{-strong})\]
which means\[
\limsup_{k\rightarrow\infty}\left\Vert p_{s}\left(x_{1,k},\ldots,x_{n,k}\right)\right\Vert =\left\Vert p_{s}\left(x_{1},\ldots,x_{n}\right)\right\Vert .\]

If $p_{s}\left(x_{1},\ldots,x_{n}\right)=0$ let $\alpha_{j,k}=1,$
and otherwise let \[
\alpha_{j,k}=\max\left(1,\left(\frac{\left\Vert p_{s}\left(x_{1},\ldots,x_{n}\right)\right\Vert }{\left\Vert p_{s}\left(x_{1,k},\ldots,x_{n,k}\right)\right\Vert }\right)^{\frac{1}{d_{s}}}\right).\]
Let $y_{j,k}=\alpha_{j,k}x_{j,k}.$ The $\alpha_{j,k}$ are are most
$1$ so $\| y_{j,k}\|\leq\| x_{j}\|.$ We are scaling by positive
factors so $0\leq y_{j,k}$ for $j\leq n_{1}$ and $y_{j,k}^{*}=y_{j,k}$
for $n_{1}<j\leq n_{2}.$ Since $\alpha_{j,k}\rightarrow1$ we see
that $y_{j,k}\rightarrow x_{j}$ in the $*$-strong topology. For
fixed $k$ the $y_{j,k}$ still all act as zero on the complement
of range of $u_{k}.$ What we have gained are the final relations,\begin{eqnarray*}
\left\Vert p_{s}\left(y_{1,k},\ldots,y_{n,k}\right)\right\Vert  & = & \left\Vert \alpha_{j,k}^{d_{s}}p_{s}\left(x_{1,k},\ldots,x_{n,k}\right)\right\Vert \\
 & = & \alpha_{j,k}^{d_{s}}\left\Vert p_{s}\left(x_{1,k},\ldots,x_{n,k}\right)\right\Vert \\
 & \leq & \epsilon_{s}.\end{eqnarray*}

\end{proof}

\section{Applications}

We have several corollaries to Theorem\ \ref{thm:RFDimpliesClosedReduction}.
All these results refer to the operator norm or order relations. Of
course, we recover the result of L\"{o}wner that matrix monotone
for all orders implies operator monotone.

\begin{cor}
Let $a$ be a bounded operator. Then\[
\left\Vert e^{a}\right\Vert \leq\left\Vert e^{\textnormal{Re}(a)}\right\Vert .\]

\end{cor}
\begin{proof}
Theorem IX.3.1 of \cite{BhatiaMatrixAnalysis} tells us this result
is true for any matrix, and indeed with any unitarily invariant norm.

Since $\left\Vert e^{x}\right\Vert \geq1$ for any operator $x,$
we can rephrase this to say that for each $\alpha\geq1,$ we have\[
\left\Vert e^{\textnormal{Re}(a)}\right\Vert \leq\alpha\implies\left\Vert e^{a}\right\Vert \leq\alpha.\]
As the first relation is RFD and the second is closed, we are done
by Theorem\ \ref{thm:RFDimpliesClosedReduction}. 
\end{proof}
The NC $*$-polynomial version (in the original variables, not their
fractional powers) of the following can be proven by Helton's sum-of-squares
theorem, Theorem\ 1.1 in \cite{McCulloughPutinarNCsumofSquares},
which is essentially in \cite{HeltonPositiveNCpoly}. 

\begin{cor}
Suppose that $p$ is a NC $*$-polynomial in $x_{1},\ldots,x_{n}$
with constant term zero, and that $t_{1},\ldots,t_{m}$ are positive
exponents and $j_{1},\ldots,j_{m}$ are between $1$ and $n.$ If
\[
p\left(x_{j_{1}}^{t_{1}},\ldots,x_{j_{m}}^{t_{m}}\right)\geq0\]
for all self-adjoint matrices $x_{1},\ldots,x_{n}$ then the same
hold true for all self-adjoint operators on Hilbert space. If only
integer powers of $x_{r}$ are used the relation $x_{r}^{*}=x_{r}$
may be dropped.
\end{cor}
\begin{proof}
The null set of relations is RFD so Theorem\ \ref{thm:RFDimpliesClosedReduction}
still applies.
\end{proof}
\begin{cor}
Let $a,$ $b$ and $x$ be a bounded operators, with $a\geq0$ and
$b\geq0$. Then, for $0\leq\nu\leq1,$\[
\left\Vert a^{\nu}xb^{1-\nu}+a^{1-\nu}xb^{\nu}\right\Vert \leq\left\Vert ax+xb\right\Vert .\]

\end{cor}
\begin{proof}
This is in \cite{BhatiaMatrixAnalysis} as Corollary IX.4.10, restricted
to matrices but for any unitarily invariant norm. Using the operator
norm version of that result, Proposition\ \ref{pro:softNCpolyRelationsAreRFD}
and Proposition\ \ref{pro:fractionalNCpolysAreClosed} we find again
the the operator result follows from the matrix result.
\end{proof}
\begin{cor}
If $C$ is a constant so that\begin{equation}
\| a\|\leq1\mbox{ and }b\geq0\implies\left\Vert ab^{\frac{1}{2}}-b^{\frac{1}{2}}a\right\Vert \leq C\left\Vert ab-ba\right\Vert ^{\frac{1}{2}}\label{eq:commutatorSquareRoot}\end{equation}
for all matrices $a$ and $b,$ then (\ref{eq:commutatorSquareRoot})
is true for all bounded operators on Hilbert space (or $C^{*}$-algebra
elements).
\end{cor}
\begin{proof}
We all hope that the constant $C=1$ works here, c.f.~\cite{BhatiaKittanehInequalitiesNormsCommutators,PedersenCommutatorInequality}.
Perhaps this reduction to the matrix case will make that easier to
prove.

We can rephrase this as\[
\| a\|\leq1,\ b\geq0,\ \left\Vert ab-ba\right\Vert \leq\delta\implies\left\Vert ab^{\frac{1}{2}}-b^{\frac{1}{2}}a\right\Vert \leq C\delta^{\frac{1}{2}}.\]
The set of relations on the left is RFD by Proposition\ \ref{pro:softNCpolyRelationsAreRFD}
and those on the right are closed by Proposition\ \ref{pro:fractionalNCpolysAreClosed}.
\end{proof}

\end{document}